\documentclass[a4,12pt]{article}

\usepackage{amsmath,amssymb}

\newtheorem{thm}{Theorem}
\newtheorem{prop}{Proposition}
\newtheorem{lem}{Lemma}
\newtheorem{cor}{Corollary}

\title{Quadratic generation of ideals \\ defining  nonsingular toric 3-folds}
\author{Shoetsu Ogata}

\begin{document}
\maketitle

\begin{abstract}
Let $X$ be a $\mathbb{P}^1$-bundle over a nonsingular toric surface which is a blowup $\mathbb{P}^2$
along at most 3 invarinant points or a blowup $\mathbb{F}_0=\mathbb{P}^1\times\mathbb{P}^1$
along at most 4 invarinant points.
We show that 
for an ample line bundle $L$ on $X$ the defining ideal of $X$ embedded by global sections of $L$  
is generated by elements of degree two.
\end{abstract}

\section*{Introduction}

Let $X$ be a projective algebraic variety and $L$ an ample line bundle on $X$.
Mumford\cite{Mf} defines $L$ {\it normally generated} if the natural homomorphism
$$
\phi: S:=\mbox{Sym} \Gamma(X,L) \to R:=\bigoplus_{k\ge0} \Gamma(X, L^{\otimes k})
$$
is surjective.  Then we define the ideal $I(X,L)$ as
$$
I(X,L) :=\mbox{Ker}\ \phi \subset S=\bigoplus_{k\ge0} S_k.
$$

Sturmfels\cite{St} conjectured in 1995 that for an ample line bundle $L$ on a nonsingular toric variety $X$
if $L$ is normally generated then the ideal $I(X,L)$ would be generated by elements of degree two.

In this article we give an affirmative answer to the conjecture in a class of nonsingular toric 3-folds.
It is given as Proposition~\ref{p2} in the section \ref{sct4}.

\begin{thm}\label{t1}
Let $X$ be a  toric  $\mathbb{P}^1$-bundle over a nonsingular toric surface
which is a blowup $\mathbb{P}^2$
along at most 3 invarinant points or a blowup $\mathbb{F}_0=\mathbb{P}^1\times\mathbb{P}^1$
along at most 4 invarinant points.
Let $L$ be an ample line bundle on $X$.  Then 
 the ideal $I(X,L)$ is generated by elements of degree two.
 \end{thm}

In the case of dimension two the conjecture is true from the result\cite{K2} of Koelman.
Koelman\cite{K1} also shows that any ample line bundle on any toric surface is normally generated.
In the dimension more than two ample line bundles are not always very ample. However, 
Demazure\cite{Dz} shows that any ample line bundle on a {\it nonsingular} toric variety is always very ample.
On the other hand, there exist ample line bundles on singular toric 3-folds which are very ample but 
not normally generated (see \cite{BG,Og3}).

Ogata\cite{Og1,Og2} shows that if $X$ is a polarized nonsingular toric 3-fold $(X,A)$ with $\Gamma(X,A+K_X)=0$, or if $X$ 
is a nonsingular toric 3-fold admits
a surjective morphism onto the projective line, then any ample line bundle $L$ on $X$ is normally generated.
Then we may try to check quadratic generation of ideals.

Ogata\cite{Og4} obtains a weak version of the conjecture such that if $(X,A)$ is a polarize toric 3-fold with
$\Gamma(X,A+K_X)=0$, then $X$ embedded by  any ample line bundle is the zero-set of quadratic binomials.

\section{Polarized toric varieties and Lattice polytopes}

In this section we recall the fact about toric varieties and ample line bundles on them, and corresponding
lattice polytopes from Oda's book\cite{Od} and Fulton's book\cite{Fu}.

Let $M\cong \mathbb{Z}^n$ be a free abelian group of rank $n$.  By scalar extension to real numbers $\mathbb{R}$
we have the real vector space $M_{\mathbb{R}}:= M\otimes_{\mathbb{Z}}\mathbb{R} \cong\mathbb{R}^n$.
The group ring $\mathbb{C}[M]$ of $M$ defines an algebraic torus $T:=\mbox{Spec}\ \mathbb{C}[M] \cong 
(\mathbb{C}^\times)^n$. The character group $\mbox{Hom}_{gp}(T,\mathbb{C}^\times)$ of $T$ is identified with $M$.
For $m\in M$ we denote the corresponding character as $e(m): T \to \mathbb{C}^\times$.

A normal algebraic variety $X$ of dimension $n$ is called toric if it has an algebraic torus action 
$T\times X \to X$ and an open orbit isomorphic to $T$, and if the action is compatible with the group action
$T\times T \to T$ and 
the inclusion $T\subset X$.

Next we define a lattice polytope. The convex hull of a subset $\{m_1, \dots, m_r\}\subset M$ in $M_{\mathbb{R}}$
$$
P:= \mbox{conv}\{m_1, \dots, m_r\} \subset M_{\mathbb{R}}
$$
is called a lattice polytope.

Let $L$ be an ample line bundle on a projective toric variety $X$ of dimension $n$.  Then there exists a lattice
polytope $P\subset M_{\mathbb{R}}$ of dimension $n$ such that
the set of lattice points $P\cap M$ corresponds to a basis of the space of global sections $\Gamma(X,L)$:
\begin{equation}\label{e1}
\Gamma(X,L) \cong \bigoplus_{m\in P\cap M} \mathbb{C} e(m).
\end{equation}
And
\begin{equation}\label{e2}
\Gamma(X, L\otimes \omega_X) \cong \bigoplus_{m\in (\mbox{\scriptsize Int}P)\cap M} \mathbb{C} e(m),
\end{equation}
 where $\omega_X$is the dualizing sheaf of $X$.
 
 Conversely, let $P\subset M_{\mathbb{R}}$ be a lattice polytope of dimension $n$ and $V(P)$ the set of vertices 
 of $P$. For every vertex $v\in V(P)$ set the convex cone $C_v(P):=\mathbb{R}_{\ge0}(P-v)$ with the apex $v$.
This defines the affine toric variety $U_v:=\mbox{Spec}\ \mathbb{C}[C_v(P)\cap M]$ of dimension $n$.
By gluing this $U_v$ all over $v\in V(P)$ we obtain a toric variety
$$
X = \bigcup_{v\in V(P)} U_v.
$$
And define the line bundle $L$ as
$$
\Gamma(U_v, L)=e(v)\mathbb{C}[C_v(P)\cap M].
$$
Then $L$ is ample and satisfies the relation (\ref{e1}).
Thus a polarized toric variety $(X,L)$ of dimension $n$ corresponds to a lattice polytope of the same dimension.

$X$ is nonsingular if and only if every cone $C_v(P)$ is nonsingular, that is, there exists a $\mathbb{Z}$-basis 
$\{m_1, \dots, m_n\}$ of $M$ such that
$$
C_v(P)=\mathbb{R}_{\ge0}m_1 + \dots +\mathbb{R}_{\ge0}.
$$

Moreover, for a natural number $k$ the natural homomorphism
$$
\Gamma(X,L)^{\otimes k} \to \Gamma(X, L^{\otimes k})
$$
is surjective if and only if the relation
$$
\overbrace{(P\cap M) + \dots +(P\cap M)}^k = (k P)\cap M
$$
holds.

\section{A class of toric 3-folds}\label{sct2}

Here we describe a class of nonsingular polarized toric 3-folds which includes our toric 3-folds
treated in this article.

\begin{thm}[Ogata\cite{Og1}]\label{t2}
Let $(X,A)$ be a polarized nonsingular toric 3-fold with $\Gamma(X,A+K_X)=0$.
Then $X$ is one of the following
\begin{itemize}
\item[(1)] A blowup $\mathbb{P}^3$ along at most 4 invariant points,
\item[(2)] A blowup the $\mathbb{P}^1$-bundle over $\mathbb{P}^2$ along at most 2 invariant points,
\item[(3)] A $\mathbb{P}^1$-bundle over a nonsingular toric surface $Y$.
\end{itemize}
Moreover, any ample line bundle on $X$ is projectively normal.
\end{thm}

We describe the shape of lattice polytopes corresponding to polarized 3-folds in the theorem.

Let $Q\subset M_{\mathbb{R}}$ be the  lattice polytope corresponding to $(X,A)$ in Theorem~\ref{t2}.
The condition $\Gamma(X,A+K_X)=0$ implies 
$$
(\mbox{Int}\ Q)\cap M=\emptyset.
$$

Let $e_1,e_2,e_3$ be a $\mathbb{Z}$-basis of $M$.
Let $\Delta_3:=\mbox{conv}\{0,e_1,e_2,e_3\}$ be the basic lattice simplex.

In the case (1), $Q$ is $k \Delta_3$ with $1\le k\le3$, $2\Delta_3$ deleted
at most 1 vertex or $3\Delta_3$ deleted at most 4 vertices.

Let $M'=\mathbb{Z}e_1+\mathbb{Z}e_2$.  Let $\Delta_2=\mbox{conv}\{0,e_1, e_2\}$.

In the case (2), $Q$ is the prism over $\Delta_2$ with edges of length $a,b,c\ge1$, or
the prism over $2\Delta_2$ with edges of length $e,f,g$ such that $e-f$ and $e-g$ are even, 
or the latter deleted one vertex in the bottom or one vertex on the roof.

In the case (3), $Q$ has two parallel polygons $F_0$ and $F_1$ which have the same number of vertices
and corresponding edges are parallel.  
Refer the Figure~\ref{fig1}.

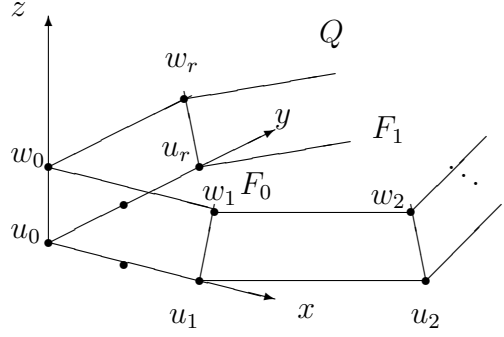
\begin{figure}[h]
 \begin{center}
 \setlength{\unitlength}{1mm}
  \begin{picture}(60,40)(20,10)
   \put(20,20){\vector(4,-1){30}}
   \put(20,20){\vector(0,1){30}}
   \put(20,20){\vector(2,1){30}}
   \put(15,20){\makebox(10,10)[bl]{$u_0$}}
   \put(53,10){\makebox(10,10)[bl]{$x$}}
   \put(15,50){\makebox(10,10)[bl]{$z$}}
   \put(50,35){\makebox(10,10)[bl]{$y$}}
   \put(40,15){\line(1,0){30}}
   \put(40,15){\line(1,5){2}}
   \put(70,15){\line(1,1){10}}
   \put(70,15){\line(-1,5){2}}
   \put(40,30){\line(6,1){20}}
   \put(40,30){\line(-1,5){2}}
   \put(15,30){\makebox(10,10)[bl]{$w_0$}}
   \put(30,5){\makebox(10,10)[r]{$u_1$}}
   \put(65,5){\makebox(10,10)[]{$u_2$}}
   \put(35,27){\makebox(10,10)[l]{$u_r$}}
   \put(35,25){\makebox(10,10)[br]{$w_1$}}
   \put(60,17){\makebox(10,10)[t]{$w_2$}}
   \put(30,35){\makebox(10,10)[tr]{$w_r$}}
   \put(50,20){\makebox(-5,15){$F_0$}}
   \put(60,30){\makebox(10,10){$F_1$}}
   \put(50,40){\makebox(15,15){$Q$}}
   \put(65,20){\makebox(20,20){$\ddots$}}
   \put(30,17){\circle*{1}}
   \put(20,30){\circle*{1}}
   \put(30,25){\circle*{1}}
   \put(20,20){\circle*{1}}
   \put(40,15){\circle*{1}}
   \put(42,24){\circle*{1}}
   \put(70,15){\circle*{1}}
   \put(68,24){\circle*{1}}
   \put(40,30){\circle*{1}}
  \put(38,39){\circle*{1}}
   \put(20,30){\line(4,-1){22}}
   \put(20,30){\line(2,1){19}}
   \put(42,24){\line(1,0){26}}
   \put(68,24){\line(1,1){10}}
   \put(38,39){\line(6,1){20}}
  \end{picture}
 \end{center}
\caption{typical $Q$ of (3)}
 \label{fig1}
\end{figure}

Toric 3-folds appeared in Theorem~\ref{t1} is in the case (3) and $Y$ is a blowup $\mathbb{P}^2$
along at most 3 invariant points, or a blowup $\mathbb{F}_0$ along at most 4 invariant points.
The corresponding lattice polygon $F_0$ is one of the following 
\begin{eqnarray*}\label{e3-1}
R_1 &=&(k \Delta_2)\cap (x+y\ge l_1) \cap (x\le l_2) \cap (y\le l_3)\\
& & \mbox{with} \quad 0\le l_1<k, l_1<l_2, l_1<l_3
\end{eqnarray*}
or
\begin{eqnarray*}\label{e3-2}
R_2 &=&\Box_{a,b} \cap (l_1\le x+y \le l_2) \cap (-k_1\le y-x\le k_2) \\
& &\mbox{with}\quad k\ge1, 
k_1\ge0, l_1\ge0, l_1<k_1, l_1<l_2, l_1<k_2, l_2+k_1>2a, l_2+k_2<2b
\end{eqnarray*}
where $\Box_{a,b}$ is the rectangle
$$
\Box_{a,b}:=[0,ae_1]\times [0,be_2] \quad \mbox{for}\quad a,b\ge1.
$$

\section{Proposition and Proof}

We state the key result of this article.

\begin{prop}\label{p1}
Let $Y$ be a blowup $\mathbb{P}^2$ along at most 3 invariant points, or a blowup $\mathbb{F}_0=
\mathbb{P}^1\times \mathbb{P}^1$ along at most 4 invariant points.
And let $X$ be a nonsingular toric 3-fold which is a $\mathbb{P}^1$-bundle over $Y$.
Assume that $A$ is an ample line on $X$ with $\Gamma(X,A+K_X)=0$.
Then the ideal $I(X,A)$ is generated by elements of degree two.
\end{prop}

For a proof of the Proposition we describe how to express an element of the ideal $I(X,A)$.
Eisenbud and Sturmfels\cite{ES} show that the ideal of a polarized toric variety $(X,L)$ is
generated by binomials. 

For a pair of lattice points $m_1, \dots, m_r$ and $m'_1, \dots, m'_r\in Q\cap M$, if it
satisfies the relation
$$
m_1+\dots +m_r = m'_1+ \dots +m'_r,
$$
then this relation defines a binomial of degree $r$ in $I(X,A)$.

Ogata\cite{Ot} and Sturmfels\cite{St} show that for a projective toric variety $X$ of dimension $n$ if
an ample line bundle $L$ on $X$ is normally generated then the ideal $I(X,L)$ is generated by elements
of degree at most $n+1$, and Ogata\cite{Ot} characterize $(X,L)$ to need a binomial of degree $n+1$ in a
generater of $I(X,L)$.  In particular, if $X$ is  nonsingular then the ideal is generated by elements of dgree
at most $n$.

In our case, the ideal $I(X,A)$ is generated by binomials of degree at most three. 
For a proof of the Proposition, from a relation of degree 3 in $Q$
\begin{equation}\label{e4}
m_1+m_2+m_3= m'_1+m'_2+m'_3
\end{equation}
we have to deform it to a relation of degree 2 by using relations of degree 2.
For example, we have to make  relations as
$$
m_1+m_2+m_3=m'_1+m''_2 +m_3 =m'_1+ m'_2+m'_3.
$$

\bigskip

{\it Proof of Proposition}.

A lattice polytope $Q$ in the case (3) has two faces $F_0$ and $F_1$ of dimension 2, and any lattice point
in $Q\cap M$ is contained in $F_0\cap M$ or $F_1\cap M$. Hence a sum $m_1+m_2+m_3\in (3Q)\cap M$
is contained one of $(3F_0)\cap M, (2F_0+F_1)\cap M, (F_0+2F_1)\cap M$ or $(3F_1)\cap M$.
If $m_1+m_2+m_3\in 3F_0$, then the relation(\ref{e4}) is the relation in $F_0$, 
thus the problem is those in the dimension two.

We assume the sum $m_1+m_2+m_3$ is contained in $2F_0+F_1$.  Moreover, lattice points in the relation
(\ref{e4}) satisfy
$$
m_1,m_2,m'_1,m'_2 \in F_0 \quad \mbox{and} \quad m_3,m'_3\in F_1.
$$

\begin{itemize}
\item Relations in $F_0$
\end{itemize}

We need an easy lemma to treat relations in dimension two.

\begin{lem}\label{l1}
Set $\Box:=\mbox{conv}\{0,(1,0),(1,1),(1,0)\}$ the basic square. Assume that $F_0$ is $R_1$ or $R_2$ 
in the previous section. For latice points $m_1,m_2\in F_0\cap M$, there exist a lattice square $S$ isomorphic to
 $\Box$ by a parallel transform of $M$ and lattice points $m''_1,m''_2 \in S\cap F_0\cap M$ such that
 the relation $m_1+m_2=m''_1+m''_2$ holds.
\end{lem}

Since $m_1+m_2=m''_1+m''_2$ is a relation of degree 2 in $Q$, we may assume that there exist lattice squares
$S_1, S_2$ such that $m_1,m_2\in S_1\cap F_0$ and $m'_1,m'_2\in S_2\cap F_0$ for lattice points in
the relation (\ref{e4}).

For lattice points $m_3,m'_3\in F_1$, we divide the situation into two cases.
\begin{itemize}
\item the case (I) : $\frac{m_3+m'_3}2 =m_0 \in M$.
\end{itemize}

In this case we note $m'_3-m_0=m_0-m_3$.  We also note the relation (\ref{e4}) means that 
the centroids of triangles of $m_1,m_2,m_3$ and $m'_1,m'_2,m'_3$ coincide.
From this we see that the vector connecting the midpoints of the segments $[m_1,m_2]$ and $[m'_1, m'_2]$
satisfies the relation
\begin{equation}\label{e5}
\frac{m_1+m_2}2 -\frac{m'_1+m'_2}2 = \frac12(m'_3-m_3) =m'_3-m_0 =m_0-m_3.
\end{equation}

Here we define a type of the midpoint as $(0,0), (0,\frac12),(\frac12,0),(\frac12,\frac12)$ according as its
coordinates are integers, or half-integers.
We note that the types of the midpoints of segments $[m_1,m_2]$ and $[m'_1,m'_2]$ coincide from 
the relation (\ref{e5}).

We divide the argument into three cases.

(a) Type $(0,0)$:
In this case $m_1=m_2$ and $m'_1=m'_2$.  From the relation (\ref{e5}) we have
$$
m_2-m'_2=m_0-m_3 \quad \mbox{and} \quad m_1-m'_1=m'_3-m_0,
$$
that is, $m_2+m_3=m'_2+m_0$ and $m_1+m_0=m'_1+m'_3$.  Thus
$$
m_1+m_2+m_3=m_1+m'_2+m_0=m'_1+m'_2+m'_3.
$$

(b) Types $(0,\frac12)$ and $(\frac12,0)$:
In this case, the segments $[m_1,m_2]$ and $[m'_1,m'_2]$ are parallel. By changing the numbers,
we have
$$
m'_1+(m'_3-m_0) =m_1 \quad \mbox{and} \quad m'_2+(m_0-m_3)=m_2,
$$
that is, $m_1+m_0=m'_1+m'_3$ and $m_2+m_3=m'_2+m_0$. 

(c) Type $(\frac12,\frac12)$:
In this case the segments $[m_1,m_2]$ and $[m'_1,m'_2]$ are parallel or perpendicular.
If they are parallel, then we get the same relations as in the case (b). We assume they are perpendicular.
We may set as
$$
m_2=m_1+(1,1) \quad \mbox{and} \quad m'_2=m'_1+(1,-1).
$$
If the segment $[m'_1,m'_2]$ after parallel transforming by $m'_3-m_0$ stays in $P$, then set
$$
m'_1+(m'_3-m_0)=m''_1 \quad \mbox{and} \quad m'_2+(m'_3-m_0)=m''_2.
$$
These make relations
$$
m'_1+m'_3=m''_1+m_0, \ m'_2+m_0=m''_2+m_3\quad\mbox{and} \quad m_1+m_2=m''_1+m''_2.
$$
Thus we have
$$
m'_1+m'_2+m'_3=m''_1+m'_2+m_0=m''_1+m''_2+m_3+m_1+m_2+m_3.
$$

Even if the segment $[m'_1,m'_2]$ after parallel transforming by $m'_3-m_0$ nor $[m_1,m_2]$ after
parallel transforming by $m_3-m_0=-(m'_3-m_0)$ are not contained in $P$, the end points of transformed segments
are contained in $P$. Set these points $\tilde{m}_1, \tilde{m}_2$ as
$$
\tilde{m}_1=m'_1+(m'_3-m_0) \ \mbox{and} \quad \tilde{m}_2=m_1-(m_0-m_3).
$$
These make relations
$$
m'_1+m'_3=\tilde{m}_1+m_0, \ m_1+m_3=\tilde{m}_2+m_0 \quad\mbox{and}\quad
 \tilde{m}_1+m'_2=m_2+\tilde{m}_2.
$$
Thus we have
$$
m'_1+m'_2+m'_3=\tilde{m}_1+m'_2+m_0=\tilde{m}_2+m_2+m_0=m_1+m_2+m_3.
$$

\begin{itemize}
\item the case (II): $\frac{m_3+m'_3}2 \notin M$
\end{itemize}

From Lemma~\ref{l1}, there exist a lattice square $S_3$ and lattice points $m_4,m'_4\in S_3\cap F_1\cap M$
such that $m_3+m'_3=m_4+m'_4$.
We have relations of vectors
$$
m_4-m_3=m'_3-m'_4 \quad \mbox{and} \quad m'_4-m_3=m'_3-m_4.
$$
We also have relations of vectors
\begin{equation}\label{e6}
\frac{m_1+m_2}2 -\frac{m'_1+m'_2}2 =\frac12(m'_3-m_3)=m'_3-\frac{m_4+m'_4}2 = \frac{m_4+m'_4}2 -m_3.
\end{equation}

We compare types of three midpoints of the segment $[m_3,m'_3]$, $[m_1,m_2]$ and $[m'_1,m'_2]$
From the relation (\ref{e6}), if the midpoint on $F_1$ is of type $(\frac12,\frac12)$, then types of the midpoints on $F_0$
are a pair of $(0,0)$ and $(\frac12,\frac12)$, or a pair of $(0,\frac12)$ and $(\frac12,0)$.
If the midpoint on $F_1$ is of type $(0,\frac12)$ or $(\frac12,0)$, then one midpoint on $F_0$ is of type $(0,0)$ or $(\frac12,\frac12)$ and another midpoint on $F_0$ is of type $(0,\frac12)$ or $(\frac12,0)$.

We may write as
$$
m'_3-\frac{m_4+m'_4}2 = (m'_3-m_4)+\frac{m_4-m'_4}2 =(m'_3-m'_4)-\frac{m_4-m'_4}2
$$
and
$$
\frac{m_1+m_2}2 -\frac{m'_1+m'_2}2 = (m_2-m'_1)+\frac{m_1-m_2}2 +\frac{m'_1-m'_2}2
=(m_1-m'_2) +\frac{m_2-m_1}2 +\frac{m'_2-m'_1}2.
$$
Comparing types of midpoints, we see that $m'_3-m_4=m_2-m'_1$ or $m'_3-m_4=m_1-m'_2$.

If $m'_3-m_4=m_2-m'_1$, then $m'_3-m'_4=m_1-m'_2$. From $m_4-m_3=m'_3-m'_4$, we have
$$
m'_1+m'_3=m_2+m_4\quad \mbox{and}\quad m_1+m_3=m'_2+m_4.
$$
Thus 
$$
m_1+m_2+m_3=m_2+m'_2+m_4 =m'_1+m'_2+m'_3.
$$

In the same way, if $m'_3-m_4=m_1-m'_2$, then $m'_3-m'_4=m_2-m'_1$.
From $m'_3-m'_4=m_4-m_3$ we have
$$
m'_2+m'_3=m_1+m_4 \quad \mbox{and} \quad m_2+m_3=m'_1+m_4.
$$
Thus 
$$
m'_1+m'_2+m'_3=m_1+m'_1+m_4=m_1+m_2+m_3.
$$

We complete the proof of Proposition~\ref{p1}.

\section{General case}\label{sct4}

We state our main result.

\begin{prop}\label{p2}
Let $Y$ be a blowup $\mathbb{P}^2$ along at most 3 invariant points, or a blowup $\mathbb{F}_0=
\mathbb{P}^1\times \mathbb{P}^1$ along at most 4 invariant points.
And let $X$ be a nonsingular toric 3-fold which is a $\mathbb{P}^1$-bundle over $Y$.
For an ample line bundle $L$ o $X$ the ideal $I(X,L)$ is generated by binomials of degree two.
\end{prop}

{\it Proof}.
Let $P\subset M_{\mathbb{R}}$ be the lattice polytope corresponding to $(X,L)$.
$P$ has two big faces $F_0$ and $F_1$ parallel to each other. By taking a basis of $M$,
assume $F_0$ sits on the $xy$-plane $(z=0)$ and $F_1$ is on the plane $(z=t>0)$.
The polygon $F_0$ is $R_1$ or $R_2$ described in the section~\ref{sct2}.
For $0\le l\le t$ set $G_l:=P\cap(z=l)$. Then $G_l$ has the same number of vertices as those of $F_0$
and corresponding edges are parallel. In particular, each $G_l$ is nonsingular.

Assume that lattice points $m_1,m_2,m_3,m'_1,m'_2,m'_3\in P\cap M$ satisfy the relation
\begin{equation}\label{e7}
m_1+m_2+m_3= m'_1+m'_2+m'_3.
\end{equation}
Consider $z$-coordinates of $m_1,m_2,m_3$.  Those of at least two are of the same parity.
By a change of numbering set $m_1, m_2$ have the same parity of those $z$-coordinates.
Then the $z$-coordinate of the midpoint $\frac{m_1+m_2}2$ is an integer $l_1$.  That is,
$$
\frac{m_1+m_2}2 \in G_{l_1}.
$$
In the same way, set $m'_1,m'_2$ have the same parity of those $z$-coordinates.

From
$$
m_1+m_2 \in (2G_{l_1})\cap M =(G_{l_1}\cap M) +(G_{l_1}\cap M),
$$
there exist lattice points $\tilde{m}_1, \tilde{m}_2\in G_{l_1}\cap M$ such that $m_1+m_2=
\tilde{m}_1+\tilde{m}_2$. Applying Lemma~\ref{l1}, there exist a lattice square $S_1\subset (z=l_1)$
and lattice points $m''_1,m''_2\in S_1\cap G_{l_1}\cap M$ such that 
$$
m_1+m_2=\tilde{m}_1+\tilde{m}_2 =m''_1+m''_2.
$$
For a proof of the proposition, we may assume $m_1,m_2 \in S_1\cap G_{l_1}\cap M$.
In the same way, we may assume that there exists a lattice square $S_2\subset (z=l_2)$ and
$m'_1,m'_2 \in S_2\cap G_{l_2}\cap M$.

From the relation (\ref{e7}), we have a relation of vectors
\begin{equation}
\frac{m_1+m_2}2 -\frac{m'_1+m'_2}2 =\frac{m'_3-m_3}2.
\end{equation}
Then we see that the $z$-coordinate of the midpoint $\frac{m_3+m'_3}2$ is an integer $l_3$.
From
$$
m_3+m'_3 \in (2G_{l_3}\cap M =(G_{l_3}\cap M) +(G_{l_3}\cap M),
$$
there exist lattice points $\tilde{m}_3, \tilde{m}'_3 \in G_{l_3}\cap M$ such that $m_3+m'_3=
\tilde{m}_3+\tilde{m}'_3$.  From Lemma~\ref{l1} again, there exist a lattice square $S_3\subset (z=l_3)$
and lattice points $m_4,m'_4\in S_3\cap G_{l_3}\cap M$ such that $\tilde{m}_3+\tilde{m}'_3=m_4+m'_4$.

By the same argument of the proof of Proposition~\ref{p1}, we complete the proof.  \hfill $\Box$

\bigskip

\begin{cor}
Let $X$ be a nonsingular toric 3-fold in the cases (1) and (2) listed in Theorem~\ref{t2}.
For an ample line bundle $L$ on $X$ the ideal $I(X,L)$ is generated by binomials of degree two.
\end{cor}
{\it Proof}.
If $X$ is the projective 3-space $\mathbb{P}^3$, the statement is trivially true.

Assume $X$ is not the projective 3-space.
Let $P\subset M_{\mathbb{R}}$ be a lattice polytope corresponding to $(X,L)$.
By a change of a basis of $M$ and using a parallel transform, we may set $P\subset (0\le z\le t)$
and each cross section $G_l:=P\cap(z=l)$ ($0\le l\le t$) is a lattice polygon.
Then $G_l$ is $R_1$ or $R_2$ described in the section \ref{sct2}.
Thus the argument of the proof of Proposion~\ref{p2} works.  \hfill $\Box$

\end{document}